\documentclass[12pt]{amsart}
\usepackage{amsmath}
\usepackage{amssymb}
\usepackage[mathcal]{eucal}
\usepackage[all]{xy}
\usepackage{latexsym}
\usepackage{amstext}
\usepackage{amsfonts}
\usepackage{amsthm}
\usepackage{amsopn}
\usepackage{amsbsy}
\usepackage{layout}
\usepackage{color}
\usepackage{graphicx}
\usepackage{bm}

\theoremstyle{plain}
\newtheorem{theorem}{Theorem}

\newtheorem{lemma}[theorem]{Lemma}
\newtheorem{corollary}[theorem]{Corollary}

\theoremstyle{definition}
\newtheorem{definition}[theorem]{Definition}
\newtheorem{problem}[theorem]{Problem}
\theoremstyle{remark}

\newtheorem{remark}[theorem]{Remark}
\newtheorem{example}[theorem]{Example}

\newcommand{\Pow}{\mathcal P}
\newcommand{\Stone}{S}

\newcommand{\dens}{\operatorname{dens}}

\newcommand{\R}{\mathbb R}
\newcommand{\N}{\mathbb N}

\begin{document}

\title[Two disjoint copies in compacta]{The two-disjoint-copies property for compact spaces, homogeneity \\
 and connection with $C_p$-theory}

\author[J. K\c akol]{Jerzy K\c akol}
\address{Faculty of Mathematics and Informatics, Adam Mickiewicz University in Pozna\'n,
61-614 Pozna\'n, Poland}
\email{kakol@amu.edu.pl}

\author[O. Kurka]{Ond\v {r}ej Kurka}
\address{Institute of Mathematics, Czech Academy of Sciences,
\v{Z}itn\'a 25, 115 67 Praha 1, Czech Republic}
\email{kurka.ondrej@seznam.cz}

\author[W.\ \'Sliwa]{Wies{\l}aw \'Sliwa}
\address{Faculty of Exact and Technical Sciences, University of Rzesz\'ow,
35-310 Rzesz\'ow, Poland}
\email{wsliwa@ur.edu.pl; sliwa@amu.edu.pl}

\date{}


\thanks{The authors thank Professor  Tomasz Kania  for  discussions and  comments  during the editing of the publication, Professor Sergey Medvedev  and Dr  Andrea Medini for valuable remarks concerning $h$-homogeneous spaces.
The second named author was supported by the Academy of Sciences of the Czech Republic (RVO 67985840).}

\begin{abstract}
A Tychonoff space $X$ has the \emph{two-disjoint-copies property} (2DCP) if there exists a sequence
$(K_n)_{n\in\omega}$ of non-empty compact subsets of $X$ such that each $K_n$ contains two disjoint subsets
homeomorphic to $K_{n+1}$. Banakh, K\c akol and \'Sliwa showed that 2DCP yields an infinite-dimensional metrizable
quotient of $C_p(X)$, while it is still a long-standing open question whether $C_p(X)$ has such a quotient  for any infinite compact space $X$.   The above concept as well as the last problem are closely  related to Efimov's problem that has remained open for 40 years. We will discuss a number of conditions  that imply 2DCP. For example, every locally homogeneous compact space,  every  space containing a copy of $\beta\omega$ or $2^\omega$ has 2DCP although compact $h$-homogeneous  spaces  with 2DCP without such copies exist in ZFC.
We prove that no scattered compact space has 2DCP  and there exist in ZFC compact  perfect spaces without 2DCP.  This implies that for compact metric spaces $X$ the  2DCP is equivalent to uncountability of $X$.
There exist    explicit uncountable separable compact spaces
failing 2DCP, for example the Isbell-Mr\'owka compacta.
We  give positive classes among zero-dimensional compact spaces; for example, the Brech, as well as the  Sobota-Zdomskyy compact spaces of Efimov type have 2DCP.  Open questions are included.
\end{abstract}

\maketitle

\section{Introduction}
Let $X$ be a Tychonoff space. By $C_p(X)$ we denote the linear space $C(X)$  of continuous real-valued functions endowed with the topology of pointwise convergence.

A line of work initiated by K\c akol and \'Sliwa \cite{Kakol-Sliwa}, and developed further by Banakh, K\c akol and \'Sliwa \cite{BanakhKakolSliwaQuotients}, asks when $C_p(X)$ admits
infinite-dimensional [metrizable] quotients. It  is still  an open question:
\begin{problem}\label{Open}
Let $X$ be a compact space. Does $C_p(X)$ admit an infinite-dimensional  [metrizable] separable quotient?
\end{problem}
If the answer to Problem \ref{Open} is negative, then the space $X$ must be an Efimov space,  while it is still unknown whether  such spaces  do exist in ZFC, see Section \ref{final} for some discussion.

By  Rosenthal's theorem,  see \cite{Rosenthal}, the corresponding Banach space $C(X)$  admits a  quotient isomorphic to the Banach spaces $l_2$ or $c_0$.

The last result is closely related to the still open problem (posed by Banach and Mazur around 1933)   whether every infinite-dimensional Banach space has an infinite-dimensional separable quotient (or equivalently, whether every infinite-dimensional Banach space can be mapped by a continuous linear mapping onto an infinite-dimensional separable Banach space),  see \cite{Mujica} also for references therein.

A useful sufficient  condition related with Problem \ref{Open}   is the following self-similarity property (due to  K\c akol--\'Sliwa) of compact subsets contained in \cite[Theorem 4]{Kakol-Sliwa}.
\begin{definition}[Two-disjoint-copies property]
\label{def:2dcp}
A Tychonoff space $X$ has the \emph{two-disjoint-copies property} (2DCP) if there exists a sequence
$(K_n)_{n\in\omega}$ of non-empty compact subsets of $X$ such that for every $n\in\omega$ the set $K_n$
contains two disjoint subsets homeomorphic to $K_{n+1}.$
\end{definition}
The concept related to 2DCP seems to be one of the possible ways to perhaps better understand Problem \ref{Open}.

Banakh, K\c akol and \'Sliwa showed that 2DCP forces infinite-dimensional metrizable quotients of $C_p(X)$, see \cite[Theorem~2]{BanakhKakolSliwaQuotients}; it is easy to see that every metrizable quotient of $C_p(X)$ is separable.
\begin{theorem}[Banakh--K\c akol--\'Sliwa]
\label{thm:bks}
If a Tychonoff space $X$ has 2DCP, then $C_p(X)$ admits an infinite-dimensional metrizable quotient.
\end{theorem}
The present note is primarily topological.
In Section \ref{First} we  gather  several  conditions  forcing $X$ to have 2DCP (Theorem \ref{I}). For example, every locally homogeneous compact space has 2DCP.
Every topological space  containing   a copy of $\beta\omega$  or a copy of the Cantor set $2^{\omega}$ has 2DCP. Nevertheless, there exist  compact (even $h$-homogeneous) spaces with 2DCP but without such copies, see Example \ref{EX}.

In Section \ref{Second}  we prove that no scattered compact space has 2DCP and show  that  there exist perfect compact spaces $K$ (in ZFC) without 2DCP (Theorem \ref{thm:scattered-no2dcp}). We construct such compact  $K$  of a similar type as in \cite{Brech2}, the so-called split Cantor space.
Consequently,  no scattered  space has 2DCP.
This unifies and strengthens several elementary counterexamples and applications, see illustrating  Examples \ref{prop:psi-basic} and \ref{M} and Corollary \ref{S1}.

In Section \ref{Second} we characterise also    2DCP for compact metric spaces: In fact the 2DCP holds exactly for the uncountable ones
(Theorem~\ref{thm:metric-characterisation}). Nevertheless, we show in Section \ref{final} (Corollary \ref{Rosen}) that a compact Rosenthal space $X$ is scattered if and only if $X$ does not admit 2DCP.

We show  in Section \ref{Second}, see  Example \ref{Br},  that the Brech \cite{Brech}  and the Sobota--Zdomskyy \cite{SobotaZdomskyy} zero-dimensional compact Efimov spaces  $K$ constructed by a method of forcing, admit 2DCP; in fact they  have topological bases of clopen sets all of them are homeomorphic to the whole space  $K$.

We give in Section \ref{homogeneity} simple sufficient conditions for 2DCP in the zero-dimensional setting
(clopen-base homogeneity and $h$-homogeneity).

In Section \ref{final} open questions and examples are also  provided.

For a space $K$, write $K'$ for the derived set, i.e. the set of non-isolated points of $K$.

Define the transfinite Cantor--Bendixson derivatives $(K^{(\alpha)})_{\alpha}$ by
$K^{(0)}=K,\,
K^{(\alpha+1)}=\bigl(K^{(\alpha)}\bigr)',
K^{(\lambda)}=\bigcap_{\alpha<\lambda} K^{(\alpha)}$  for limit ordinal  $\lambda.$

A topological space $K$ is called a \emph{scattered} space if every non-empty subspace of $K$ has an isolated point. By the Cantor--Bendixson Theorem (\cite[8.5.2 Theorem]{SemadeniCk}), $K$ is scattered if and only if $K^{(\alpha)}=\varnothing$ for some ordinal $\alpha$.

If $K$ is a scattered space, let $h(K)$ be the least ordinal $\alpha$ with $K^{(\alpha)}=\varnothing$.
It is easy to check the following  statements  (see for example  \cite{EngelkingGT} and \cite{SemadeniCk}).
\begin{remark} \label{star}
Let $X$ and $Y$ be non-empty topological spaces.

(1) If $X$ is a subspace of $Y$, then $X^{(\alpha)} \subset Y^{(\alpha)}$ for every ordinal $\alpha.$

(2) If $X$ and $Y$ are homeomorphic, then $X^{(\alpha)}$ and $Y^{(\alpha)}$ are homeomorphic for every ordinal $\alpha.$

(3)  If $X$ and $Y$ are scattered and $X$ is embedded into $Y$, then $h(X)\leq h(Y)$.

(4) If $X$ is scattered and compact, then $h(X)=\beta +1$ for some ordinal $\beta$, and $X^{(\beta)}$ is non-empty and finite.
\end{remark}
Recall   \cite[Theorem 8.5.4]{SemadeniCk} that  compact  $X$ is scattered if and only if there is no continuous surjection from $X$ onto $[0,1]$ if and only if dim $X=0$ and there is no continuous surjection from $X$ onto $2^{\omega}$.

By $(c_0)_p$ and $(l_{\infty})_p$ we denote the sequence spaces $c_0$ and $l_{\infty}$, respectively, with the topology induced from $\R^{\N}$.
We propose also the following  fundamental question connected with Problem \ref{Open}.
\begin{problem}\label{XX}
Assume that a compact space $X$ has 2DCP. Does the space $C_p(X)$ admit a quotient isomorphic to $(l_{\infty})_p$ or $(c_0)_p$?
\end{problem} Note that the spaces  $C_p(\beta\omega)$ and $C_p(2^{\omega})$ admit such quotients, see \cite[Theorem 1]{BanakhKakolSliwaQuotients} and \cite[Theorem 1]{Banakh-Kakol-Sliwa}, respectively. Consequently, if a compact space $X$ contains a copy of $\beta\omega$ or $2^\omega$, then applying the restriction mappings we infer that $C_p(X)$ has a quotient isomorphic to $(l_{\infty})_p$ or $(c_0)_p$.
\section{Preliminary facts and Proof of Theorem \ref{I}}\label{First}
Let $X$ be a topological space. For any $x\in X$ we denote by $J(x)$ the set of all maps $\phi$ such that there exists an open neighbourhood $V$ of $x$ such that $\phi: V\rightarrow X$ is injective and continuous.
The set $O(x) = \{\phi (x) :\phi\in J(x)\}$ is called the orbit of $x$.
   A subset  $V\subset X$  is  called a {\it quasi-neighbourhood} of a point $x\in X$ if it is homeomorphic to some open neighbourhood of $x$.

In \cite{Kakol-Sliwa} we discussed some  classes of spaces  admitting   2DCP.   Clearly any Tychonoff space containing a copy of $\beta\omega$ or $2^\omega$ has 2DCP.  Indeed, the last space is zero-dimensional $h$-homogeneous \cite[Lemma 1.9.5]{van Mill} and Theorem \ref{I}(3) applies. Trivially $\beta\omega$ has 2DCP.

The following Theorem \ref{I}  completes these classes of spaces.
\begin{theorem}\label{I}
An infinite compact space $X$ has  2DCP if it satisfies at least one of the following conditions:
\begin{enumerate}
\item The orbit $O(x)$ of every point $x\in X$ is infinite.

\item The set $\{\phi(x): \phi \in \Phi\}$ is infinite for every $x\in X$, where $\Phi$ is the family of all continuous injective  maps from $X$ to $X$.

\item  $X$ is locally homogeneous, i.e.~any two points $x, y \in X$ have open neighbourhoods $V_x$ and $V_y$, such that  there exists a homeomorphism $\varphi: V_x \to V_y$ with $\varphi (x)=y.$

\item For every $x\in X$ and every $n\in \omega$ the space $X$ contains $n$ pairwise disjoint quasi-neighbourhoods of $x$.

\item $X$ contains infinitely many pairwise disjoint quasi-neighbourhoods of any point $x\in X$.

\item Every non-empty open subset $U$ of $X$ contains a quasi-neighbourhood of any point $x\in X$.

\item The complement $(X\setminus F)$ of any finite subset $F$ of $X$ contains a quasi-neighbourhood of any $x\in F$.

\item The complement $(X\setminus F)$ of any finite subset $F$ of $X$ contains a copy of $X$.

\item Every non-empty open subset of $X$ contains two disjoint non-empty homeomorphic open subsets.

\item $X$ contains two disjoint copies of itself.

\item $X$ is extremally disconnected.

\end{enumerate}
\end{theorem}
\begin{proof} (1) By \cite[Theorem 12 and Lemma 13]{Kakol-Sliwa} and their proofs we get that (1) implies 2DCP.

(2) implies (1), since $\Phi \subset J(x)$ for every $x\in X.$

By (3) we get $O(x)=X$ for every $x\in X$, so (3) implies (1).

(4) implies (1). Indeed, let $x\in X$. Let $n \in \omega$ and let $W_1, \ldots, W_n$ be pairwise disjoint quasi-neighbourhoods of $x$. Let $j_k: W_k \to X$ be the inclusion map for $1\leq k \leq n$. For every $1\leq k \leq n$ there exists an open neighbourhood $U_k$ of $x$ and a homeomorphism $\varphi_k: U_k \to W_k$. Then $$\psi_k:=j_k\circ \varphi_k \in J(x), \,\,1\leq k \leq n$$ and $$\psi_k(x) \neq \psi_l(x),\,\, 1\leq k,l \leq n$$ with $k\neq l$. Thus $|O(x)|\geq n$ for any $n\in \omega$, so $O(x)$ is infinite.

Clearly, (5) implies (4).

(6) implies (5), since $X$ contains infinitely many of pairwise disjoint non-empty open subsets.

(7) implies (1). Indeed, suppose by contrary that the orbit $O(x)$ of some $x$ is finite. Clearly, $x\in O(x)$, so the set $(X\setminus O(x))$ contains a quasi-neighbourhood $W$ of $x$. Thus there exists an open neighbourhood $U$ of $x$ and a homeomorphism $\varphi: U \to W$; consequently  $$\varphi \in J(x),\,\,\,\varphi(x) \in W\cap O(x)=\emptyset.$$ This provides a  contradiction.

Clearly, (8) implies (7).

(9) To prove 2DCP it is enough to show that every compact subset $K$ of $X$ with non-empty interior in $X$, contains two homeomorphic disjoint compact subsets with non-empty interiors in $X$. Let $U_1, U_2$ be two homeomorphic disjoint non-empty open subsets of Int$(K)$. Let $V_1$ be a non-empty open subset of $U_1$ with $\overline{V_1} \subset U_1$ and let $\varphi: U_1 \to U_2$ be a homeomorphism. Then $\overline{V_1}$ and $\varphi (\overline{V_1})$ are homeomorphic disjoint compact subsets of $K$ with non-empty interiors in $X$.

Clearly, (10) implies 2DCP.

(11) implies 2DCP, since every infinite extremally disconnected compact space contains a copy of $\beta \omega$ which has 2DCP.
\end{proof}
Consequently, every locally homogeneous compact space admits  2DCP. Note that there exist locally homogeneous compact spaces not being homogeneous, see Example \ref{T} below.
\section{two-disjoint-copies property  and scatteredness}\label{Second}
We are ready to prove the following  Theorem \ref{thm:scattered-no2dcp} below.  The argument for the first claim is elementary but nontrivial in the sense that it gives a clean obstruction: Indeed,  2DCP enforces a binary branching in the last non-empty Cantor-Bendixson derivative, which
is impossible because that derivative is finite for scattered compacta.

The first part of  Theorem \ref{thm:scattered-no2dcp} suggests  a fundamental question  whether  there exists in ZFC  a non-scattered compact space $K$  without 2DCP.  Surely such a space cannot contain copies of $\beta\omega$ or  $2^\omega$; however, it can contain a nontrivial convergent sequence. This last condition essentially distinguishes Efimov's problem, see Section \ref{final}. We construct
the compact perfect  space $K$  of a  similar  type as in \cite{Brech2};  so-called split Cantor space.
\begin{theorem}\label{thm:scattered-no2dcp}
No scattered compact Hausdorff space has 2DCP. On the other hand, there  exists in ZFC  a perfect compact space $K$ without 2DCP.
\end{theorem}
\begin{proof}
Suppose by contradiction that there is a sequence $(K_n)_{n\in\omega}$ as in Definition \ref{def:2dcp}.
Put $\alpha_n=h(K_n)$ for $n\in\omega$.
Since $K_{n+1}$ embeds into $K_n$, then by Remark \ref{star} we get $\alpha_{n+1}\leq \alpha_n$ for $n\in\omega$. Let $\alpha_p$ be the least element of the set $\{\alpha_n: n\in\omega\}$.
Then $\alpha_n=\alpha_p$ for every $n\geq p$. By Remark \ref{star}, $\alpha_p=\beta +1$ for some ordinal $\beta$ and $K_n^{(\beta)}$ is non-empty and finite for every $n\geq p$. Let $n\geq p$. Let $C_{n,0}$ and $C_{n,1}$ be disjoint compact subsets of $K_n,$ that are homeomorphic to $K_{n+1}$. By Remark \ref{star}, $C_{n,0}^{(\beta)}$ and $C_{n,1}^{(\beta)}$ are disjoint compact subsets of $K_n^{(\beta)},$ that are homeomorphic to $K_{n+1}^{(\beta)}$. Thus
\[ |K_n^{(\beta)}|\ \geq \ |C_{n,0}^{(\beta)}|+|C_{n,1}^{(\beta)}|\ =\ 2\,|K_{n+1}^{(\beta)}|,\] for every $n\geq p.$
It follows that $$|K_p^{(\beta)}|\geq 2^m |K_{p+m}^{(\beta)}|\geq 2^m$$ for all $m\in\omega$; a contradiction. Thus $X$ does not have 2DCP.

Now we prove the other part of theorem: Let
$$ \mathcal{Z} = \big\{ (A, B, \gamma) : \textrm{$ A, B \subset 2^{\omega} $ countably infinite, $ \gamma : A \to B $ bijective} \big\}. $$
Then there is an enumeration
$$ \mathcal{Z} = \big\{ (A_{\xi}, B_{\xi}, \gamma_{\xi}) : 0 \leq \xi < \mathfrak{c} \big\}. $$
Recursively, for $ 0 \leq \xi < \mathfrak{c} $, we want to find
\begin{itemize}
\item $ r_{\xi}, s_{\xi} \in 2^{\omega} $ different from each other and from $ r_{\xi'}, s_{\xi'}, \xi' < \xi $,
\item a continuous $ f_{r_{\xi}} : 2^{\omega} \setminus \{ r_{\xi} \} \to 2 = \{ 0, 1 \} $ that does not have a limit in $ r_{\xi} $,
\item a sequence $ a_{\xi}^{n} \in A_{\xi} \setminus \{ r_{\xi} \}, n \in \omega $,
\end{itemize}
satisfying the requirements
\begin{itemize}
\item[{(i)}] $ a_{\xi}^{n} \to r_{\xi} $ as $ n \to \infty $,
\item[{(ii)}] $ \gamma_{\xi}(a_{\xi}^{n}) \to s_{\xi} $ as $ n \to \infty $,
\item[{(iii)}] $ f_{r_{\xi}}(a_{\xi}^{n}) $ does not converge for $ n \to \infty $.
\end{itemize}
It may happen that it is not possible to find $ r_{\xi}, s_{\xi}, f_{r_{\xi}} $ and $ a_{\xi}^{n} $ satisfying (i)--(iii). In such a case, we do not pose these requirements.

The resulting space $ K $ will be the space $ 2^{\omega} $ in which each $ x $ of the form $ x = r_{\xi} $ is replaced by two points $ x^{0} $ and $ x^{1} $ with their neighbourhoods intersected by $ \{ f_{x} = 0 \} $ and $ \{ f_{x} = 1 \} $.

To provide a precise definition, let us put $ \mathcal{R} = \{ r_{\xi} : 0 \leq \xi < \mathfrak{c} \} $. Let us denote by $ B_{\delta}(x) $ the open ball around $ x $ with radius $ \delta > 0 $ (for a fixed compatible distance on $ 2^{\omega} $). We equip
$$ K = \{ x^{0} : x \in \mathcal{R} \} \cup \{ x^{1} : x \in \mathcal{R} \} \cup \{ x : x \in 2^{\omega} \setminus \mathcal{R} \} $$
with the topology given by the neighbourhoods
$$ U(x, \delta) = \Pi^{-1}({B_{\delta}(x)}), \quad x \in 2^{\omega} \setminus \mathcal{R}, \; \delta > 0, $$
$$ U(x^{\tau}, \delta) = \{ x^{\tau} \} \cup \Pi^{-1} \big( \{ p \in B_{\delta}(x) \setminus \{ x \}: f_{x}(p) = \tau \} \big), \quad x \in \mathcal{R};$$  $\tau \in 2, \delta > 0,$
where we define
$$ \Pi(x) = x, \; x \in 2^{\omega} \setminus \mathcal{R}, \quad \Pi(x^{0}) = \Pi(x^{1}) = x, \; x \in \mathcal{R}. $$

Let us note first that $ K $ does not have isolated points. For $ x \in 2^{\omega} \setminus \mathcal{R} $, it is clear that $ U(x, \delta) $ contains points different from $ x $. For $ x \in \mathcal{R} $ and $ \tau \in 2 $, we use that $ f_{x} : 2^{\omega} \setminus \{ x \} \to 2 $ does not have a limit in $ x $, so we can find a sequence $ p_{n} $ converging to $ x $ such that $ f_{x}(p_{n}) = \tau $.

\textbf{Claim:} $ K $ does not contain two disjoint copies of a non-scattered compact space $ L $.

Indeed, assume that $ L_{1}, L_{2} \subset K $ are disjoint non-scattered and $ h : L_{1} \to L_{2} $ is a homeomorphism. Note that $ \Pi(L_{1}) $ and $ \Pi(L_{2}) $ are compact, as $ \Pi $ is continuous. We can assume that $ L_{1} $ and $ L_{2} $ are perfect, and so are $ \Pi(L_{1}) $ and $ \Pi(L_{2}) $. Both $ \Pi(L_{1}) $ and $ \Pi(L_{2}) $ have cardinality continuum, and the same follows for $ L_{1} $ and $ L_{2} $.

We can also assume that $ \Pi(L_{1}) $ and $ \Pi(L_{2}) $ are disjoint, because $ \Pi(L_{1}) \cap \Pi(L_{2}) $ must be finite (otherwise there is an accumulation point $ x $ of $ \Pi(L_{1}) \cap \Pi(L_{2}) $, which satisfies $ x \in L_{1} \cap L_{2} $ if $ x \notin \mathcal{R} $, respectively, $ x^{0} \in L_{1} \cap L_{2} $ or $ x^{1} \in L_{1} \cap L_{2} $ if $ x \in \mathcal{R} $).

Let $ C \subset \Pi(L_{1}) $ be a dense countable subset. For each $ c \in C $, we choose some $ \eta(c) \in L_{1} \cap \Pi^{-1}(c) $. Then $ \eta(C) $ is dense in $ L_{1} $ (to show for instance that $ U(x^{0}, \delta) $ intersects $ \eta(C) $, where $ x^{0} \in L_{1} $, we use that $ U(x^{0}, \delta) \setminus \{ x^{0} \} $ intersects $ L_{1} $, thus the open set $$ \{ p \in B_{\delta}(x) \setminus \{ x \}: f_{x}(p) = 0 \} $$ intersects $ \Pi(L_{1}) $, and so it intersects $ C $). Hence, $ h(\eta(C)) $ is dense in $ L_{2} $. We further put $ B = \Pi(h(\eta(C))) $, which is a countable dense subset of $ \Pi(L_{2}) $. For each $ b \in B $, we choose some $$ \sigma(b) \in h(\eta(C)) \cap \Pi^{-1}(b).$$ Then $ \sigma(B) $ is dense in $ L_{2} $ (by the same argument as above), hence $ h^{-1}(\sigma(B)) $ is dense in $ L_{1} $. Finally, let $$ A = (\eta^{-1} \circ h^{-1} \circ \sigma)(B),$$ which is dense in $ \Pi(L_{1}) $, and let $$ \gamma = (\eta^{-1} \circ h^{-1} \circ \sigma)^{-1} = \sigma^{-1} \circ h \circ \eta|_{A}.$$

Now, there is $ \xi < \mathfrak{c} $ such that $$ (A, B, \gamma) = (A_{\xi}, B_{\xi}, \gamma_{\xi}).$$  The set $$ M_{\xi} = \Pi^{-1}(\{ r_{\xi'} : \xi' < \xi\} \cup \{ s_{\xi'} : \xi' < \xi\}) $$ has cardinality less than continuum, thus we can choose $$ w \in L_{1} \setminus (M_{\xi} \cup h^{-1}(L_{2} \cap M_{\xi})).$$  Then $ r = \Pi(w) $ and $ s = \Pi(h(w)) $ are different from $ r_{\xi'}, s_{\xi'}, \xi' < \xi $. Also, $ r \neq s $, since $ \Pi(L_{1}) $ and $ \Pi(L_{2}) $ are disjoint. Let $ w_{n} $ be a sequence in $ h^{-1}(\sigma(B)) = \eta(A) $ converging to $ w $ with $ \Pi(w_{n}) \neq \Pi(w) $, and let $ a_{n} = \eta^{-1}(w_{n}) $. Then $$ a_{n} = \eta^{-1}(w_{n}) = \Pi(w_{n}) $$ converges to $ \Pi(w) = r $, and $$ \gamma(a_{n}) = (\sigma^{-1} \circ h)(w_{n}) = \Pi(h(w_{n})) $$ converges to $ \Pi(h(w)) = s $.

It is not difficult to show that there is a continuous $ f : 2^{\omega} \setminus \{ r \} \to 2 $ such that $ f(a_{n}) $ does not converge for $ n \to \infty $. Indeed, let $ V_{0} = 2^{\omega} $ and, for $ k = 0, 1, 2, \dots $, let $ V_{k+1} \subset V_{k} $ be a clopen neighbourhood of $ r $ such that $ V_{k+1} \subset B_{1/(k+1)}(r) $ and $ V_{k} \setminus V_{k+1} $ contains $ a_{n} $ for some $ n $. We define $ f(x) = 0 $ for $$ x \in \bigcup \{ V_{k} \setminus V_{k+1} : \textrm{$ k $ is even} \} $$ and $ f(x) = 1 $ for $$ x \in \bigcup \{ V_{k} \setminus V_{k+1} : \textrm{$ k $ is odd} \}.$$

Then $ r, s, f $ and $ a_{n} $ satisfy (i)--(iii), and it follows that $ r_{\xi}, s_{\xi}, f_{r_{\xi}} $ and $ a_{\xi}^{n} $ satisfy (i)--(iii).

As $ s_{\xi} \notin \mathcal{R} $, by (ii), the sequence $$ \sigma(\gamma(a_{\xi}^{n})) = h(\eta(a_{\xi}^{n})) $$ converges to $ s_{\xi} $ in $ K $. In particular, the sequence $ h(\eta(a_{\xi}^{n})) $ is convergent in $ L_{2} $, and so the sequence $ \eta(a_{\xi}^{n}) $ is convergent in $ L_{1} $. As $ r_{\xi} \in \mathcal{R} $, by (i), $ \eta(a_{\xi}^{n}) $ converges to $ (r_{\xi})^{\kappa} $ for some $ \kappa \in 2 $. But then $ f_{r_{\xi}}(a_{\xi}^{n}) $ converges to $ \kappa $, which contradicts (iii). The proof of Claim is completed.

Finally, we show that $ K $ does not have 2DCP. Assume that $ K $ has 2DCP and that it is witnessed by a sequence $ (K_{n})_{n\geq 0} $. Then $ L = K_{1} $ has 2DCP, which is witnessed by $ (K_{n})_{n \geq 1} $, and it follows that it is not scattered. As $ K $ contains a copy of $ K_{0} $ and $ K_{0} $ contains two disjoint copies of $ L = K_{1} $, we obtain a contradiction with the claim.
\end{proof}
The last part of (the proof of)   Theorem \ref{thm:scattered-no2dcp}    yields
\begin{corollary}\label{SCA}
No scattered  space $X$ has 2DCP.
\end{corollary}
The scattered nature of a compact space $X$ appears in a number of important theorems of functional analysis, for example, the Namioka-Phelps theorem, see  \cite[Theorem 18]{Phelps},   states that a Banach space $C(X)$ has the Asplund property if and only if the compact space $X$ is scattered. Using Theorems \ref{I}  and \ref{thm:scattered-no2dcp}, we obtain the following practical
\begin{corollary}\label{S1}  Let $X$ be a compact space with 2DCP. Then $C(X)$ is not Asplund, i.e., there exists a continuous convex function $f$ on $C(X)$ such that the set of all points of Fr\'echet differentiability of $f$ is not dense (but it is always  a $G_{\delta}$-set by \cite[Lemma 8.23]{Fabian}).
\end{corollary}
Recall briefly the  Isbell-Mr\'owka  space $\Psi(\mathcal A)$.
 Let $\mathcal{A}$ be an almost disjoint family of subsets of  natural numbers $\mathbb{N}$ and let $\Psi(\mathcal{A})$ be the set $\mathbb{N} \cup \mathcal{A}$
equipped with the topology defined as follows. For each $n \in \mathbb{N}$, the singleton $\{n\}$ is open, and
for each $A \in \mathcal{A}$, a base of neighbourhoods of $A$ is the collection of all sets of the form
$\{A\} \cup B$, where $B\subset A$ and $|A \setminus B| < \omega$.
If $\mathcal{A}$ is a maximal almost disjoint family, then  $\Psi(\mathcal{A})$ is pseudocompact, separable and locally compact,  see  also \cite[Chapter 8]{Hernandez} for several  properties of these spaces.
 Now, denote by $X=\Psi(\mathcal A)^*$ the one-point compactification of $\Psi(\mathcal{A})$.
 Then $h(X)=3.$
\begin{example}\label{prop:psi-basic}
$X=\Psi(\mathcal A)^*$ is uncountable, separable, scattered compact space, and has Cantor--Bendixson height $3$:
In particular, $X$ fails 2DCP.
\end{example}
Note that the statement from \cite[Proposition 4.6]{Leiderman} before Corollary \ref{SCA}  seems to be the best optimal. Indeed, in \cite[Remark 5.3]{Leiderman} the authors provided an  example of a scattered pseudocompact locally compact space $X$ for which
$\beta X$ is not scattered.
On the other hand, we have also the following
\begin{example}\label{M}  There exists  a pseudocompact locally compact scattered Isbell-Mr\'owka space $Y=\Psi(\mathcal A)$ such that $\beta Y$  has 2DCP.
\end{example}
\begin{proof}
There exists a maximal almost  disjoint family  $\mathcal{A}$ of subsets of  natural numbers $\mathbb{N}$ such that the Isbell-Mr\'owka space $Y=\Psi(\mathcal A)$ has the property  that $(\beta Y\setminus Y)$ is homeomorphic to the interval $[0,1]$ (see \cite[Theorem 8.6.2]{Hernandez}). In particular, $\beta Y$  is not scattered. Clearly  $Y$ is pseudocompact, locally compact and scattered. Since $(\beta Y\setminus Y)$  is a closed subspace of  $\beta Y$  and $(\beta Y\setminus Y)$ has 2DCP (as homeomorphic to $[0,1]$), the space $\beta Y$ has 2DCP as claimed.
\end{proof}
\begin{remark}\label{rem:psi-Cp}
Although $X=\Psi(\mathcal A)^*$ fails 2DCP, the space $C_p(X)$ admits infinite-dimensional metrizable quotients:
$X$ contains many non-trivial convergent sequences, and we apply  \cite[Theorem 1]{Banakh-Kakol-Sliwa}.
Thus 2DCP is sufficient but not necessary for metrizable quotients of $C_p(X)$.
\end{remark}
On the other hand, applying Theorem \ref{thm:scattered-no2dcp} we have the following characterization of 2DCP for compact metric spaces.
\begin{theorem}\label{thm:metric-characterisation}
Let $K$ be a compact metric space. The following assertions are equivalent:
\begin{enumerate}
\item $K$ has 2DCP;
\item $K$ is uncountable;
\item $K$ contains a  homeomorphic  copy of the Cantor set $2^\omega$.
\end{enumerate}
\end{theorem}
\begin{proof}
(2)$\Leftrightarrow$(3) is classical: Recall that  every uncountable Polish space contains a copy of $2^\omega$, see \cite[Corollary 6.5]{Kechris}.

(3)$\Rightarrow$(1): If $C\subset K$ is homeomorphic to the space $2^\omega$, then $C$ contains two disjoint clopen subsets each homeomorphic to $2^\omega$;
for instance the sets of binary sequences starting with $0$ and with $1$.
Taking the sequence $K_n=C$ for all $n$ witnesses 2DCP.

(1)$\Rightarrow$(2): If $K$ were countable compact, then the space $K$ is scattered, see \cite[Proposition 8.5.7]{SemadeniCk}, contradicting our
Theorem~\ref{thm:scattered-no2dcp}.
\end{proof}
We provide also the following Example \ref{Br}  which is strongly motivated by Brech's \cite[Theorem 3.1]{Brech}. Brech  used  the  method of forcing  to show  that consistently there is a
Banach space of continuous functions on a compact Hausdorff space with the
Grothendieck property and with density less than the continuum.
\begin{theorem}[Brech]
Assume  \(V\models\mathrm{CH}\), let \(\kappa>\omega_1\) be a regular cardinal,
and let \(G\) be \(\mathbb S_\kappa\)-generic over \(V\), where \(\mathbb S_\kappa\) is
the product of \(\kappa\) copies of Sacks forcing. Work in the extension \(V[G]\).
Consider the Boolean algebra $B:=\Pow(\omega)\cap V$
(with the usual Boolean operations). Let  $K:=\Stone(B)$ be its Stone space.
In \(V[G]\), the Banach space \(C(K)\) is Grothendieck and  $\dens C(K)=\omega_1<\mathfrak c=\kappa .$
\end{theorem}
\begin{example}\label{Br}
Brech's compact space $K$ does not contain copies of the space  $\beta\omega$ and does not have infinite converging sequences but yet $K$ admits  2DCP.  Note here that our Example \ref{last} in Section \ref{final} provides (but under $\diamondsuit$) compact Efimov spaces with(out) 2DCP.
\end{example}
\begin{proof}
It is obvious that $K$ does not admit infinite converging sequences; otherwise the Banach space $C(K)$ would contain complemented copies of the Banach space $c_0$ (by applying the main Cembranos result of \cite{Ce}). Moreover, since $\dens C(K)=\omega_1<\mathfrak c$, there is no continuous  surjection from $C(K)$ onto the Banach space  $\ell_{\infty}$ (whose density is $\mathfrak{c}$). This implies that $K$ does not contain  copies of $\beta\omega$.

Now we prove that $K$ has 2DCP.

For \(a\in B\), let
$B\restriction a:=\{b\in B:b\subseteq a\}$
be the relative Boolean algebra below \(a\).

\textbf{Claim (A)}: For every \(a\in B\), the clopen subset \(\widehat a\subseteq K\) is homeomorphic to
\(\Stone(B\restriction a)\).

Indeed, define
$\Phi:\widehat a\longrightarrow \Stone(B\restriction a), \;
\Phi(u)=u\cap (B\restriction a)$.
Since \(a\in u\) for every \(u\in\widehat a\), the family \(\Phi(u)\) is an ultrafilter on
\(B\restriction a\).

Conversely, if \(v\in \Stone(B\restriction a)\), then
\[
\Psi(v):=\{b\in B: b\cap a\in v\}
\]
is an ultrafilter on \(B\), and \(a\in\Psi(v)\), so \(\Psi(v)\in\widehat a\).

It is routine to check that \(\Phi\) and \(\Psi\) are inverse to each other. Moreover,
for each \(b\in B\restriction a\),
\[
\Phi[\widehat b\cap \widehat a]
=
\{v\in \Stone(B\restriction a): b\in v\},
\]
so \(\Phi\) is a homeomorphism. This proves Claim (A).

\textbf{Claim (B)}: The compact space \(K\) contains two disjoint clopen copies of itself.

Indeed,  let
$$A_0:=2\omega=\{2n:n\in\omega\},\,\, A_1:=2\omega+1=\{2n+1:n\in\omega\}.$$
Since \(A_0,A_1\in V\), we have \(A_0,A_1\in B\).
They are disjoint, non-zero, and \(A_0\cup A_1=\omega\).

For \(i\in\{0,1\}\), define a bijection
$$e_i:\omega\to A_i,\qquad e_0(n)=2n,\quad e_1(n)=2n+1.$$
Because \(e_i\in V\), the map
$$\varphi_i:B\longrightarrow B\restriction A_i,\qquad
\varphi_i(X)=e_i[X]$$
is a Boolean isomorphism.
Hence, by Stone duality,
$\Stone(B\restriction A_i)\cong \Stone(B)=K.$

By Claim (A)  the clopen set \(\widehat{A_i}\subseteq K\) is homeomorphic
to \(\Stone(B\restriction A_i)\), and therefore
$\widehat{A_i}\cong K\; (i=0,1).$
Finally,
$\widehat{A_0}\cap \widehat{A_1}=\varnothing,$ because no ultrafilter can contain two disjoint elements of the Boolean algebra.
Thus \(K\) contains two disjoint clopen subspaces, each homeomorphic to \(K\).
This proves Claim (B), and consequently the proof of Example \ref{Br} is completed.
\end{proof}
\begin{remark}
(1)  A similar argument applies to show that $K$ admits a  base of clopen sets each of them is homeomorphic to the  space $K$. (2) In \cite{SobotaZdomskyy} Sobota and Zdomskyy, motivated by Brech's result, studied the preservation of the Nikodym property by the Sacks forcing  \(\mathbb S_\kappa\) and proved that if $\mathcal{A}$ is a
$\sigma$-complete Boolean algebra in $V$, then $\mathcal{A}$ has the Nikodym property in the \(\mathbb S_\kappa\)-generic extension $V[G]$. The same proof as above shows that the Stone space  $K:=\Stone(\mathcal{A})$ admits  2DCP for $\mathcal{A}=P(\omega)$. We refer also to recent \cite{ManevSobotaZdomskyy} for a similar result of that type with more properties.
\end{remark}
\section{$h$-homogeneity and 2DCP}\label{homogeneity}
The notion of $h$-homogeneity was introduced by Ostrovskii  and van Mill, and next studied systematically by Terada \cite{Terada} and others,
 see \cite{ArhangelskiiVanMill},  \cite{vanEngelen},  \cite{MedvedevHomogeneity}, \cite{CarroyMediniMuller}, \cite{MediniProducts}.

 A topological space $X$  is \emph{homogeneous} if for each   $x, y\in X$ there exists a homeomorphism $f: X\rightarrow X$  with $f(x) = y$, and $X$ is  \emph{$h$-homogeneous} if every
non-empty clopen subset of  $X$ is homeomorphic to $X$.

The space $\beta\omega\setminus\omega$ is $h$-homogeneous but not homogeneous.  Van Douwen, see   \cite{van Douwen},   constructed a compact homogeneous space which is not $h$-homogeneous.
Terada \cite{Terada} posed the problem (still being open) whether $X^\omega$ is $h$-homogeneous for  zero-dimensional first-countable spaces $X$. Medini (\cite[Theorem 3.1]{Medini2}) proved that $X^{\lambda}$  is $h$-homogeneous for every zero-dimensional non-separable metrizable space $X$ and every infinite cardinal $\lambda$.
Many infinite powers are $h$-homogeneous under mild hypotheses; see \cite{Terada,MediniProducts}.
For every non-empty zero-dimensional space $X$ there exists a non-empty zero-dimensional space $Y$ such that $X\times Y$ is $h$-homogeneous, see \cite[Corollary~16]{MediniProducts}.
\begin{corollary}\label{cor:h-hom}
Every infinite topological space containing an infinite compact disconnected  $h$-homogeneous space $X$ has 2DCP.
\end{corollary}
\begin{proof}
Since $X$ is disconnected, there exists a separation $X=U\cup V$ where $U$ and $V$
are non-empty disjoint open subsets of $X$. As $U=X\setminus V$ and $V=X\setminus U$,
both $U$ and $V$ are clopen in $X$.
By $h$-homogeneity, every non-empty clopen subset of $X$ is homeomorphic to $X$. Now we apply Theorem \ref{I}(10).
\end{proof}
\begin{problem}
Does every infinite $h$-homogeneous compact space have 2DCP?
\end{problem}
The search for an infinite compact
$h$-homogeneous space without 2DCP reduces to the general problem of finding a nontrivial connected
compact space without 2DCP.

The Cantor set $2^\omega$ is zero-dimensional $h$-homogeneous and  $\beta\omega$  is not $h$-homogeneous although has 2DCP.
On the other hand, we note  the following
\begin{example}
There exist metrizable $h$-homogeneous spaces $X$ without copies of  $2^{\omega}$  for which $\beta X$ has  2DCP.
\end{example}
\begin{proof}
In \cite[Theorem 3.3]{Medini} was constructed a topological dense subgroup $X$ of $\mathbb{Z}^{\omega}$ which   is also a $\lambda$-set (i.e. every countable subset of $X$ is $G_{\delta}$). Clearly $X$ is zero-dimensional,  and as  a dense subgroup of  $\mathbb{Z}^{\omega}$, is not locally precompact (i.e. there is no   precompact
neighbourhood of the identity in $X$). Hence, by applying \cite[Theorem 8]{MedvedevHomogeneity}, see also \cite[Theorem 3.1]{Medini}, we deduce that $X$ is $h$-homogeneous.   Since $X$ is a $\lambda$-set, it cannot contain copies of $2^{\omega}$.  On the other hand, it is well-known that  every non-compact metrizable space contains a closed discrete copy of $\omega$; hence $\beta X$ contains a copy of $\beta\omega$, so $\beta X$ has 2DCP.
\end{proof}
\section{Final remarks and  illustrating examples}\label{final}
Since no perfect space is scattered, next fundamental  Lemma  follows directly  from \cite[Theorem 8.5.4]{SemadeniCk}.
\begin{lemma}\label{lem:onto-cantor}
For every perfect compact zero-dimensional Hausdorff space $X$  there exists a continuous surjection from $X$ onto $2^\omega$.
\end{lemma}
Lemma \ref{lem:onto-cantor} is often combined with elementary  methods producing
metrizable continuous images. However, 2DCP does not behave well under continuous
images, preimages or closed subspaces.
\begin{example}\label{III}
A topological space with 2DCP may have a continuous image without 2DCP  (since every compact metric space is a continuous
image of Cantor set).  Conversely, applying  example in Theorem \ref{thm:scattered-no2dcp}  one gets that  there exist perfect
zero-dimensional compact spaces admitting a surjection onto $2^{\omega}$  but failing 2DCP.
\end{example}
\begin{example}\cite[Example 10(3)]{kakol-molto}
There exist compact zero-dimensional nonseparable spaces whose every infinite separable closed subspace contains a copy of $\beta\omega$, so it lacks infinite converging sequences.
\end{example}
On the other hand, using Theorem \ref{thm:metric-characterisation} we note that no infinite compact metric space is hereditarily 2DCP.

Many consistent constructions related to Efimov spaces live in the totally disconnected world.
For compact Hausdorff spaces, total disconnectedness implies zero-dimensionality (the components coincide with quasi-components).
In that setting, 2DCP often follows from clopen self-similarity.
Recall that an infinite compact Hausdorff space $K$ is an \emph{Efimov space} if it contains neither a copy of $\omega+1$
(a non-trivial convergent sequence) nor a copy of $\beta\omega$.

It is an open question  whether Efimov spaces exist in ZFC; see Hart's survey \cite{HartSurveyEfimov} and
D\v{z}amonja--Plebanek \cite{DzamonjaPlebanekEfimov}, or \cite[Example 17]{Kakol-Sliwa} and for several references therein.
Under additional axioms, there are Efimov spaces both with and without 2DCP.
\begin{example}\label{last}
Under $\diamondsuit$ there are totally disconnected hereditarily separable Efimov spaces which are $h$-homogeneous
(and hence have 2DCP), see \cite[Theorem 3.22]{DelaVega} and \cite[Example 15]{Kakol-Sliwa},   and also totally disconnected hereditarily separable Efimov spaces which are perfect but admit no two disjoint
infinite closed homeomorphic subspaces (and hence fail 2DCP); see \cite[Corollaries~6.23(2) and~6.24]{KakolSobotaZdomskyyGrothJNP}.
\end{example}
\begin{remark}\label{RE} In \cite{Talagrand} Talagrand constructed (under (CH)) an Efimov space $X$ such that the Banach space $C(X)$  does not contain complemented copies of $c_0$ and yet does not have any quotient isomorphic to $\ell_{\infty}$.
Consequently, applying the closed graph theorem,  the space $C_p(X)$ does not have  quotients isomorphic to $(\ell_{\infty})_p$.

Moreover, the space $C_p(X)$ does not have quotients isomorphic to   $(c_0)_{p}$. Indeed,
otherwise, using \cite[Theorem 1]{Banakh-Kakol-Sliwa} the space $C_p(X)$ would have a complemented copy of $(c_0)_p$. The closed graph theorem would show that $C(X)$ would have a complemented  copy of $c_0$, a contradiction.
\end{remark}
This Remark \ref{RE},  Example \ref{last} combined with
Problem \ref{Open} and Problem \ref{XX}  suggest the following essential
\begin{problem} Does the  space $X$ constructed by Talagrand have 2DCP?
\end{problem}
Being motivated by these remarks about Efimov spaces and the topic around 2DCP one can ask
if  there exists a compact   space in ZFC with 2DCP without  copies of $\beta\omega$ or  $2^{\omega}$.

\emph{YES does exist!} We answer this  question in the affirmative. The following Example \ref{EX} was  suggested by  A. Medini who  provided to the authors the proof  that  the classical double arrow space is $h$-homogeneous.

Recall that the double arrow space is the space  $$X=\{(x,0):0<x\le1\}\cup\{(x,1):0\le x<1\}$$  with the lexicographic order \((x,0)<(x,1)<(x',0)\) for \(x<x'\), up to two isolated points $(0,0)$ and $(1,1)$ that we removed from $X$.
\begin{example} \label{EX}
The double arrow space $X$ equipped with the order topology is a zero-dimensional compact $h$-homogeneous space (hence has 2DCP by Corollary \ref{cor:h-hom}) which does not  contain a copy of $\beta\omega$ or  a copy of $2^{\omega}$ although $X$ has  infinite converging sequences.  Moreover,
\begin{enumerate}
\item $C_p(X)$  has a complemented subspace isomorphic to  $(c_0)_p$.
\item  There is no continuous linear surjection from $C_p(X)$  onto $(\ell_{\infty})_p$. In particular, $C_p(X)$ has no quotient isomorphic to $(\ell_{\infty})_p$.
\end{enumerate}
\end{example}
\begin{proof}
 It is well-known  that the space  $X$ is a compact first-countable zero-dimensional space, see \cite[Example 3.10.C]{EngelkingGT}.
Moreover, observe that
$\{[x^+,y^-]:x,y\in [0,1]\text{ and }x<y\}$
is a $\pi$-base for $X$ consisting of clopen sets. Furthermore, if $x,y\in [0,1]$ and $x<y$, then $[x^+,y^-]$ is clearly order-isomorphic (hence homeomorphic) to $X$. Since $X$ is  first-countable, it follows from a result of Matveev (see \cite[Proposition 17]{MediniProducts} for a slightly more general result) that $X$ is $h$-homogeneous. Clearly $X$ is zero-dimensional. Hence $X$ has 2DCP by Corollary \ref{cor:h-hom}.

Since $|X|=\mathfrak{c}$, the space  $X$ does not contain a copy of $\beta\omega$.
Observe also that $X$ cannot contain copies of $2^{\omega}$. This follows directly from the following well-known fact: \emph{All metrizable subspaces of $X$ are countable}.

Since $X$ contains the  convergent sequence space $S$, we apply \cite[Theorem 1]{BanakhKakolSliwaQuotients} to get the  claim  that  $(c_0)_p (\approx C_p(S))$ is complemented in $C_{p}(X)$.

Now assume that there is a continuous linear surjective  mapping $T: C_p(X)\rightarrow (\ell_{\infty})_p$. It is easy to see that $T: C(X)\rightarrow \ell_{\infty}$ has closed graph between Banach spaces $C(X)$ and $\ell_{\infty}$.
Hence, applying the closed graph theorem, see \cite[Theorem 2.26]{Fabian}, we get that $T: C(X)\rightarrow l_{\infty}$  is continuous. Since the size of the  dual of $\ell_{\infty}$ is $2^{\mathfrak{c}}$ while for the dual of $C(X)$ is $\mathfrak{c}$, see \cite{Kenderov},  we have  a contradiction.
\end{proof}
In \cite{van Douwen} van Douwen constructed a compact homogeneous space $bH$ which is not  $h$-homogeneous  which is hereditary  Lindel\"of (see the remark preceding
\cite[Fact 5.1]{van Douwen}), hence $bH$  does not contain copies of $\beta\omega$ and
 $bH \setminus H$   does not contain $2^\omega$, see \cite[Fact 5.1]{van Douwen}.

Recall that  a compact space $X$ is Rosenthal compact, if $X$ can be embedded into the space $B_{1}(Y)$ of Baire-one functions on a Polish space $Y$ with the pointwise topology. Trivially every compact metric space is Rosenthal. Next Corollary \ref{Rosen} supplements Theorem \ref{thm:metric-characterisation}.
\begin{corollary}\label{Rosen}
A compact Rosenthal space $X$ is not scattered if and only if $X$ admits 2DCP.
\end{corollary}
\begin{proof}
If $X$ is scattered, we apply Theorem \ref{thm:scattered-no2dcp}. Conversely, if $X$ is not scattered, then by \cite[Theorem 4]{Rosenthal} the space $X$ either contains the Cantor set $2^{\omega}$ or the double arrow space. Both copies admit 2DCP, so $X$ has the same property.
\end{proof}
Theorem \ref{I} (3) may suggest the following
\begin{example}\label{T}
 Let $X = S^2\sqcup T^2$ be the disjoint union of the $2$-sphere and the $2$-torus (with the disjoint union topology). Then $X$
is compact and locally homogeneous (hence has 2DCP by Theorem \ref{I}), but it is not homogeneous.
\end{example}
\begin{proof}
Clearly $X$ is compact. Note that  $X$ is locally homogeneous. Indeed, fix $x, y\in  X$. Since both $S^2$ and $T^2$  are $2$-manifolds (surfaces)
without boundary, each point has an open neighbourhood homeomorphic to $\mathbb{R}^2$.  Let $$\phi_{x} :
V_x\rightarrow \mathbb{R}^2,\,\, \phi_{y}: V_y \rightarrow \mathbb{R}^2$$ be such homeomorphisms with $\phi_x(x)=0$ and   $\phi_y(y)= 0$ (by
composing with a translation if necessary). Then the map $$f =\phi_{y}^{-1}\circ\phi_{x}: V_x\rightarrow V_y$$ is a
homeomorphism such that $f(x)=y$.  Observe also that $X$  is not homogeneous. In fact, the connected components of $X$ are $S^2$ and $T^2.$ Any
homeomorphism $F: X\rightarrow X$  must map a connected component onto a connected component.
If $X$ were homogeneous, there would exist a homeomorphism $F$ mapping a point $x\in S^2$ to
a point $y\in T^2$.  This would imply $F(S^2)=T^2$, so $S^2\approx T^2.$ This is a contradiction (since $\pi_1(S^2)=0$ and $\pi_1(T^2)=\mathbb{Z}^2$).
\end{proof}

The following simple observation motivates next Corollary \ref{AA}:  Let $Y\subset X$ be a compact countable infinite space and assume $X$ is compact metrizable uncountable. Clearly the restriction map $T: C_p(X)\rightarrow C_p(Y)$ is continuous linear (and open) but $Y$ fails 2DCP although $X$ has 2DCP. Nevertheless, we have
\begin{corollary}\label{AA}
Let $X$ and $Y$ be  compact spaces. Then $X$ is metrizable with 2DCP if and only if $Y$ is metrizable with 2DCP provided at least one of the conditions below holds:
\begin{enumerate}
\item $C_p(X)$ and $C_p(Y)$ are homeomorphic.
\item There are continuous linear surjections $T:C_p(X)\rightarrow C_p(Y)$ and $Q:C_p(Y)\rightarrow C_p(X)$.
\end{enumerate}
\end{corollary}
\begin{proof} Assume that $C_p(X)$ and $C_p(Y)$ are homeomorphic and $X$ is metrizable with 2DCP. Then  $X$ is uncountable by Theorem \ref{thm:metric-characterisation} and  $C_p(X)$ and $C_p(Y)$ are  separable.  Hence $Y$ is submetrizable, so it is metrizable.  If $Y$ were countable, then $C_p(Y)$ and $C_p(X)$ would be metrizable, which implies the countability of $X$, a contradiction. Hence $Y$ is uncountable, so it has 2DCP (by Theorem \ref{thm:metric-characterisation}).

Assume that item (2) holds. Then $X$ is metrizable if and only if $Y$ is metrizable. If $X$ has 2DCP, then $X$ is uncountable.  Assume that $Y$ is countable. Hence must be scattered.  By the assumption (and using the closed graph theorem) we note that  $Q: C(Y)\rightarrow C(X)$ is a continuous open linear surjection. By \cite[Theorem 18]{Phelps} each Banach space $C(Z)$ is Asplund if and only if $Z$ is scattered. Hence, since $C(Y)$ is Asplund, the space $C(X)$ is also Asplund \cite[Theorem 11]{Phelps}. This implies that $X$ is countable,  a contradiction. Hence $Y$ has 2DCP. The converse implication is clear.
\end{proof}
The spaces  $C_p([0,1]^\omega)$ and $C_p([0,1])$ are not homeomorphic, see \cite{Cauty}, but both compact spaces $[0,1]^{\omega}$ and $[0,1]$ have 2DCP.
 While 2DCP concerns compact subsets, it is conceptually related to (much stronger) self
similarity conditions on open sets.

A Hausdorff space $X$ is \emph{openly homogeneous} if every non-empty open
subset of $X$ is homeomorphic to $X$. It is openly rigid if no two distinct non-empty open
subsets of $X$ are homeomorphic.

 Classical openly homogeneous examples include Baire spaces $\kappa^{\omega}$  (for
infinite discrete $\kappa$), where every non-empty open set is homeomorphic to $\kappa^{\omega}$; see
 Brian–-Miller \cite{Brian}. Examples of openly rigid continua and dendrites
have been constructed; see Hart-–van Mill \cite{Hart-2}  and Charatonik–-Charatonik \cite{Charatonik}.

\end{document}